\documentclass[12pt, reqno]{amsart}

\usepackage{mymacros,oitp}
\usepackage[all]{xy}
\usepackage{fullpage}
\usepackage[pagebackref]{hyperref}

\usepackage{arydshln}

\title{Degenerations of orbifold curves as noncommutative varieties}

\author{Tarig Abdelgadir}
\address{Mathematical Sciences \\
Loughborough University,
LE11 3TU,
United Kingdom
}
\email{t.abdelgadir@lboro.ac.uk}

\author{Daniel Chan}
\address{School of Mathematics and Statistics, 
UNSW Sydney, 
NSW 2052,
Australia
}
\email{danielc@unsw.edu.au}

\author{Shinnosuke Okawa}
\address{
Department of Mathematics,
Graduate School of Science,
Osaka University,
Machikaneyama 1-1,
Toyonaka,
Osaka,
560-0043,
Japan.
}
\email{okawa@math.sci.osaka-u.ac.jp}

\author[K.~Ueda]{Kazushi Ueda}
\address{
Graduate School of Mathematical Sciences,
The University of Tokyo,
3-8-1 Komaba,
Meguro-ku,
Tokyo,
153-8914,
Japan.}
\email{kazushi@ms.u-tokyo.ac.jp}

\begin{document}

\begin{abstract}
Boundary points
on the moduli space of pointed curves
corresponding to collisions of marked points
have modular interpretations as degenerate curves.
In this paper,
we study degenerations of orbifold projective curves
corresponding to collisions of stacky points
from the point of view of noncommutative algebraic geometry.
\end{abstract}

\maketitle


%
%
\section{Introduction}\label{sc:introduction}

A noncommutative projective variety
in the sense of Artin--Zhang~\cite{MR1304753}
is the Grothendieck category
$\Qgr A$
obtained as the quotient
of the category $\Gr A$
of graded right modules
over a graded associative algebra $A$
satisfying the condition $\chi_1$
by the full subcategory $\Tor A$
consisting of torsion modules.


In~\cite{Abdelgadir-Okawa-Ueda_nccubic},
a class of AS-regular algebras
determined by the quiver in~\pref{fg:dP3_quiver}
describing noncommutative cubic surfaces
are introduced.
The moduli stack
of this class of algebras
is an open substack
of the quotient of an affine space
by a linear action of a torus.
A choice of a stability condition
gives an 8-dimensional smooth proper toric stack,
which contains the moduli stack of smooth marked cubic surfaces
as a locally closed substack.
An interesting feature of this moduli stack
is that $\Qgr$ of the algebra
associated with the point on the boundary
corresponding to a degenerate cubic surface
with a rational double point
is the category of quasi-coherent sheaves
not on the singular cubic surface
but on a noncommutative crepant resolution
(which is derived-equivalent
to the minimal resolution
by the McKay correspondence),
and hence is of finite homological dimension.
\begin{figure}
  \begin{tikzpicture}[thick,>=stealth,baseline=(current  bounding  box.west)]
    \graph [nodes={draw, circle}, grow right=2.5cm, branch down=2cm]
    { {00,10,20} ->[complete bipartite] {01,11,21} -> [complete bipartite] {02,12,22} };
  \end{tikzpicture}
  \caption{The quiver describing noncommutative cubic surfaces}\label{fg:dP3_quiver}
\end{figure}
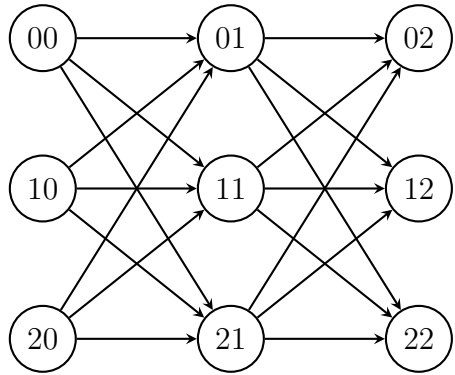

\begin{figure}
  \begin{tikzpicture}[->,thick,>=stealth,vertex/.style={circle,fill=white,draw},baseline=(current  bounding  box.west),cross/.style={preaction={-,draw=white,line width=3mm}}]
      \node[vertex] (0) at (0,0) {\tiny 0};
      \node[vertex] (11) at (3,3) {{\tiny 1,1}};
      \node[vertex] (12) at (6,3) {{\tiny 1,2}};
      \node (13) at (9,3) {$\cdots$};
      \node[vertex] (14) at (12,3) {{\tiny 1,$r_1$-1}};
      \node[vertex] (21) at (3,1) {{\tiny 2,1}};
      \node[vertex] (22) at (6,1) {{\tiny 2,2}};
      \node (23) at (9,1) {$\cdots$};
      \node[vertex] (24) at (12,1) {{\tiny 2,$r_1$-1}};
      \node (31) at (3,-1) {$\vdots$};
      \node (32) at (6,-1) {$\vdots$};
      \node (33) at (9,-1) {$\vdots$};
      \node (34) at (12,-1) {$\vdots$};
      \node[vertex] (41) at (3,-3) {{\tiny $n$,1}};
      \node[vertex] (42) at (6,-3) {{\tiny $n$,2}};
      \node (43) at (9,-3) {$\cdots$};
      \node[vertex] (44) at (12,-3) {{\tiny 2,$r_n$-1}};
      \node[vertex] (1) at (15,0) {\tiny 1};
      \draw (0) -- node[pos=.5,fill=white] {$a_{11}$} (11);
      \draw (11) -- node[pos=.5,fill=white] {$a_{12}$} (12);
      \draw (12) -- node[pos=.5,fill=white] {$a_{13}$} (13);
      \draw (13) -- node[pos=.5,fill=white] {$a_{1,r_1-1}$} (14);
      \draw (14) -- node[pos=.5,fill=white] {$a_{1,r_1}$} (1);
      \draw (0) -- node[pos=.5,fill=white] {$a_{21}$} (21);
      \draw (21) -- node[pos=.5,fill=white] {$a_{22}$} (22);
      \draw (22) -- node[pos=.5,fill=white] {$a_{23}$} (23);
      \draw (23) -- node[pos=.5,fill=white] {$a_{2,r_2-1}$} (24);
      \draw (24) -- node[pos=.5,fill=white] {$a_{2,r_2}$} (1);
      \draw (0) -- node[pos=.5,fill=white] {$a_{n1}$} (41);
      \draw (41) -- node[pos=.5,fill=white] {$a_{n2}$} (42);
      \draw (42) -- node[pos=.5,fill=white] {$a_{n3}$} (43);
      \draw (43) -- node[pos=.5,fill=white] {$a_{n,r_n-1}$} (44);
      \draw (44) -- node[pos=.5,fill=white] {$a_{n,r_n}$} (1);
  \end{tikzpicture}
  \caption{The quiver describing orbifold projective lines}\label{fg:orbifold projective line}
  \end{figure}
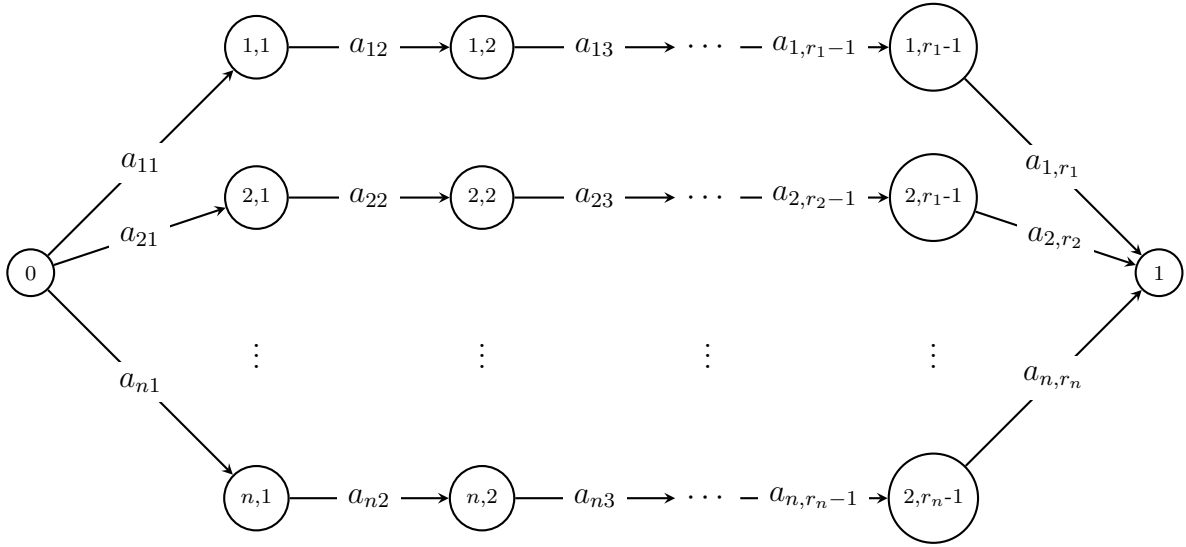

In this paper,
we explore similar phenomena in dimension one.
As a direct one-dimensional analog
of (not necessarily commutative) del Pezzo surfaces,
we consider orbifold projective lines
studied
in depth
in~\cite{MR915180}.
Given a positive integer $n$,
a sequence
$
\bsr=(r_1,\ldots,r_n)
$
of integers greater than $1$,
and a sequence
\begin{align}
    \bslambda
    =
    (
        \lambda_1
        =
        \infty,
        \lambda_2
        =
        0,
        \lambda_3
        =
        1,
        \lambda_4,
        \ldots,
        \lambda_n
    )
    \in
    \left.
        \left(
            \left(
                \bP^1
            \right)
            ^n
            \setminus
            \Delta
        \right)
        \middle/
        \PGL(2)
    \right.
\end{align}
of pairwise distinct points on $\bP^1$
where

\begin{align}
    \Delta
    \coloneqq
    \left\{
        (
            \lambda_i
        )
        _{i=1}^n
        \in
        \left(
            \bP^1
        \right)
        ^n
        \relmid
        \lambda_i
        =
        \lambda_j
        \text{ for some }
        i \ne j
    \right\}
\end{align}
is the big diagonal,
an \emph{orbifold projective line}
is defined by

\begin{align}\label{eq:orbifold projective line}
  \bfX
  =
  \bfX_{\bsr,\bslambda}
  \coloneqq
  \left[ \left( \Spec S \setminus \bszero \right) \middle/ L^\dual \right],
\end{align}
where

\begin{align}\label{eq:S}
    S
    =
    S
    _{
        \bsr,
        \bslambda
    }
    \coloneqq
    \left.
    \bfk[
        x_1,\ldots,x_n
        ]
    \middle/
    \left(
        x_i^{r_i} - x_2^{r_2} + \lambda_i x_1^{r_1}
    \right)
    _{i=3}^n
    \right.
\end{align}
is the homogeneous coordinate ring of $\bfX$
and
$
L^\dual \coloneqq \Spec \bfk[L]
$
is the Cartier dual
of the abelian group

\begin{align}\label{eq:L}
  L = L_\bsr \coloneqq
  \left.
  \bZ \vecc \oplus
  \bigoplus_{i=1}^n \bZ \vecx_i
  \middle/
  \langle
    \vecc - r_i \vecx_i
    \rangle
    _{i=1}^n
  \right.
\end{align}
of rank one.
The coaction
$
S \to \bfk[L] \otimes S
$
dual to the action
$
L^\dual \times \Spec S \to \Spec S
$
is given by
$
x_i \mapsto \vecx_i \otimes x_i
$
for any
$
i \in \{1,\ldots,n\}.
$
The morphism
$
c \colon \bfX \to \bP^1
$
to the coarse moduli scheme
is the $r_i$-th root construction at $\lambda_i$
in a neighborhood of $\lambda_i \in \bP^1$
for each $i$,
and the Picard group $\Pic \bfX$
can be identified with $L$
in such a way that
$\cO_\bfX \left( \vecx_i \right)$ is the universal bundle
associated with the root construction,
so that
$
\cO_{\bfX} \left( \vecc \right)
= \cO_\bfX \left( r_i \vecx_i \right)
\simeq \cO_\bfX(1)
\coloneqq c^{ \ast } \cO_{\bP^1}(1).
$
The dualizing element
\begin{align}
  \vecomega \coloneqq
  (n-2) \vecc - \vecx_1 - \cdots - \vecx_n
\end{align}
corresponds to the canonical sheaf of $\bfX$.
We equip $L$ with a structure of an ordered set
so that
$\veca \preceq \vecb$
if
\begin{align}\label{eq:def of preceq}
  \vecb- \veca \in L^{\ge 0} \coloneqq \left\{ \sum_{i=1}^n a_i \vecx_i \relmid a_i \in \bZ^{\ge 0} \right\}.
\end{align}
Note that
one has
$
\veca \preceq \vecb
$
if and only if
\begin{align}
    \Hom
    \left(
        \cO_\bfX \left( \veca \right),
        \cO_\bfX \left( \vecb \right)
    \right)
    \ne
    0.
\end{align}

The endomorphism algebra of the full strong exceptional collection
\begin{equation}\label{eq:Geigle-Lenzing FSEC}
    \left(
        \cO_\bfX \left( \veca \right)
    \right)
    _{
        0
        \preceq
        \veca
        \preceq
        \vecc
    }
\end{equation}
of line bundles given in~\cite[Proposition 4.1]{MR915180}
is described by the quiver
$
    Q = Q_\bsr
$
in~\pref{fg:orbifold projective line}
with relations
\begin{equation}\label{eq:canonical relations}
    \cI
    =
    \cI_{\bsr,\bslambda}
    =
    \vspan \left\{ a_{i r_i} \cdots a_{i 1}
    -
    a_{1 r_1} \cdots a_{1 1}
    +
    \lambda_i a_{2 r_2} \cdots a_{2 1} \relmid 3 \le i \le n \right\},
\end{equation}
which is a two-sided ideal
of the path algebra $\bfk Q$
contained in $\bfe_1 \bfk Q \bfe_0$
as a vector subspace.

From the ideal $\cI$ of relations of the quiver $Q$,
one can recover the ideal of relations
of the homogeneous coordinate ring~\pref{eq:S}
as the ideal generated by the image of $\cI$
by the linear map
\begin{align}
  \bfe_1 \bfQ \bfe_0 \to
  \bfk
  \left[
    x_1,\ldots, x_n
  \right]
  \qquad
  a_{i r_i} \cdots a_{i 1} \mapsto x_i^{r_i},
  \quad i \in \{ 1, \ldots, n \}.
\end{align}
This allows the construction of $\bfX$
as the fine moduli stack of \emph{refined} representations
of the quiver
\begin{align}\label{eq:Gamma}
  \Gamma = (Q,\cI)
\end{align}
with relations~\cite{MR3274027}.

The \emph{moduli stack of relations},
discussed more generally in~\cite{1411.7770},
is defined as the stack quotient
\begin{align}
\scrR \coloneqq \left[ \Gr_{n-2} \left( \bfe_1 \bfk Q \bfe_0 \right) \middle/ \Gm^{Q_1} \right]
\end{align}
of the Grassmannian of $(n-2)$-dimensional subspaces
of the $n$-dimensional vector space $\bfe_1 \bfk Q \bfe_0$
by the action of the torus $\Gm^{Q_1}$
coming from the change of basis of $\bfk a$
for each arrow $a \in Q_1$.

\begin{remark}
The stack $\scrR$ has a non-trivial group of generic automorphisms,
which can be rigidified by taking the quotient by
$
    (\Gm)
    ^{n-1}
$
instead of $\Gm^{Q_1}$.
\end{remark}

If we write
\begin{align}
\bsr
=
\left(
\overbrace{ { r }_{ 1 }, \dots, { r }_{ 1 } } ^{ d _{ 1 } },
\overbrace{ { r }_{ 2 }, \dots, { r }_{ 2 } } ^{ d _{ 2 } },
\dots,
\overbrace{ { r }_{ \ell }, \dots, { r }_{ \ell } } ^{ d _{ \ell } }
\right)
\end{align}
for pairwise distinct integers
$
r _{ 1 }, \dots, r _{ \ell },
$
then the automorphism group of the quiver $Q$ is given by
\begin{align}
  \Aut Q \simeq \frakS_{\bsd} \coloneqq \prod_{i=1}^\ell \frakS_{d_i}.
\end{align}
The isomorphism class of an orbifold projective line $\bfX$
determines the full strong exceptional collection~\pref{eq:Geigle-Lenzing FSEC},
which in turn determines the isomorphism class of the quiver $(Q,\cI)$ with relations.
Therefore,
the subspace
$
\cI \subset \bfe_1 \bfk Q \bfe_0
$
is determined by the isomorphism class of $\bfX$
up to the action of
$
\Gm^{Q_1} \rtimes \Aut Q.
$

In order to find 
a well-behaved open substack of the non-separated stack $\scrR$,
replace $\Gr(n-2,n)$ with $\Gr(2,n)$
and
use the Gelfand--MacPherson correspondence~\cite{MR658730}
\begin{align}
  \Gr(2,n) /\!/ (\Gm)^{n-1}
  &\simeq
  \left( \GL_2 \backslash \! \backslash \Mat(2,n) \right) /\!/ (\Gm)^{n-1} \\
  &\simeq
  \PGL_2 \backslash \! \backslash \left( \Mat(2,n) /\!/ (\Gm)^n \right) \\
  &\simeq
  \PGL_2 \backslash \! \backslash \left( \bP^1 \right)^n.
\end{align}
A stability condition
in the sense of geometric invariant theory~\cite{MR1304906}
is determined
by a \emph{stability parameter}
$
\chi = (\chi_i)_{i=1}^n \in \bZ^n \cong \Pic \left( \bP^1 \right)^n
$
in such a way that
a point
$
\bslambda = (\lambda_i)_{i=1}^n \in \left( \bP^1 \right)^n
$ 
is \emph{$\chi$-semistable}
if for any $I \subset \{1,\ldots, n\}$
such that $\lambda_i = \lambda_j$ for $i,j \in I$,
one has
\begin{align}
  \sum_{i \in I} \chi_i \le \frac{1}{2} \sum_{i=1}^n \chi_i.
\end{align}
The point $\bslambda$ is \emph{$\chi$-stable}
if the above inequality is strict.
A stability parameter $\chi$ is \emph{generic}
if semistability implies stability.
For a generic stability parameter,
the open substack $\scrR^{\mathrm s} \subset \scrR$
consisting of stable points
is a principal $B (\Gm)^{|Q_1|-n+1}$-bundle
over a smooth projective scheme.

We now ask if the points
on the boundary of $\scrR^{\mathrm s}$
(i.e., those which do not come from
$
\left( \bP^1 \right)^n \setminus \Delta
$)
have `geometric' interpretations.

To give a `commutative' answer to this problem,
first note that
the definitions of $S_{\bsr,\bslambda}$
and $\bfX_{\bsr,\bslambda}$
in~\pref{eq:S} and~\pref{eq:orbifold projective line}
make sense for any
$
(\lambda_4,\ldots,\lambda_n) \in \bA^{n-3},
$
and can be interpreted as the fine moduli stack
of refined representations of $\Gamma$
in~\pref{eq:Gamma}.
The resulting stack is described as follows:

\begin{theorem}\label{th:fiber product}
    For any
    $
        (\lambda_4,\ldots,\lambda_n) \in \bA^{n-3},
    $
    one has an isomorphism
    \begin{align}\label{eq:fiber product}
        \bfX_{\bsr,\bslambda}
        \simeq
        \bfX_{r_1,\lambda_1}
        \times_{\bP^1}
        \bfX_{r_2,\lambda_2}
        \times_{\bP^1}
        \cdots
        \times_{\bP^1}
        \bfX_{r_n,\lambda_n}
    \end{align}
    of stacks.
\end{theorem}

To search for a `noncommutative' answer
to the same problem,
note that a stability condition in geometric invariant theory
comes from a choice of a ray
in the group of characters.
In the world of stacks,
we must specify not only a ray
but a submonoid,
since
taking the Veronese subring changes $\Qgr$ in general.
In the noncommutative world,
we can choose a subset of the group of characters
which may not be closed under addition
(so that the resulting `noncommutative scheme'
may not be a quotient by a group action).

Consider the $\bfk$-linear category
\(
    \sfS = \sfS_{\bsr,\bslambda}
\)
whose set of objects is $L$
and whose space of morphisms
from $\veca$
to $\vecb$
is the space $S_{\vecb-\veca}$
of homogeneous elements
of degree $\vecb - \veca$.
For a subset $K \subset L$
containing $\bZ \vecc$
and invariant under translation by $\vecc$,
we set
\begin{align}\label{eq:sfSK}
\sfS_K
\coloneqq
(\text{the full subcategory of $\sfS$
consisting of $K$}).
\end{align}
A right $\sfS_K$-module
is a contravariant functor
from $\sfS_K$
to $\Module \bfk$.
An $\sfS_K$-module $M$ is said to be \emph{torsion}
if for any $m \in M$,
there exists $\veca \in K$
such that $m s = 0$
for any $\vecb \in K$
satisfying $\vecb \succeq \veca$
and any $s \in S_{\vecb}$.
The quotient
of the Grothendieck category
$\Gr \sfS_K$ of right $\sfS_K$-modules
by the localizing subcategory $\Tor \sfS_K$
consisting of torsion modules
will be denoted by $\Qgr \sfS_K$.
As is common in noncommutative algebraic geometry
(cf.~e.g.~\cite[Section 3.5]{MR1846352}),
we write $\Qcoh \bfX_K \coloneqq \Qgr \sfS_K$
and talk about quasi-coherent sheaves on $\bfX_K$,
although the symbol $\bfX_K$ alone does not make any sense.
We write the image in $\Qcoh \bfX_K$
of the projective module $\bfe_{\veca} \sfS_K$
represented by the object $\veca \in L$
as $\cO_{\bfX_K} \left( \veca \right)$.
For $K = L$,
one has
\begin{align}
  \bfX_L \simeq \bfX,
\end{align}
and for $K = \bZ \vecc$,
one has
\begin{align}
  \bfX_{\bZ \vecc} \simeq \bP^1.
\end{align}

A flag
$
K' \subset K \subset L
$
of subsets gives a full embedding
$
\sfS_{K'} \to \sfS_{K},
$
the restriction
$
\Gr \sfS_{K} \to \Gr \sfS_{K'}
$
along which induces a functor
$
\Qgr \sfS_{K} \to \Qgr \sfS_{K'}
$
since a torsion module restricts to a torsion module.
This functor is exact and can be regarded
as the push-forward along the `noncommutative contraction'
$\bfX_{K} \to \bfX_{K'}$.

Let $\varpi \colon \bfX \to \bP^1$
be the noncommutative contraction
corresponding to the inclusion
$
\bZ \vecc \subset L,
$
and
\begin{align}\label{eq:A_XK}
  \cA_{\bfX_K} &\coloneqq \varpi_{ \ast } \cEnd \left( \bigoplus_{\veca \in K_0} \cO_{\bfX} \left( \veca \right) \right)
\end{align}
where
\begin{align}
  K_0
  &\coloneqq
  \left\{
    \veca \in K
  \relmid
    0 \preceq \veca \prec \vecc
  \right\}.
\end{align}
We say that
\(
K _{ 0 }
\)
is a fundamental domain of the translation by
\(
\vecc
\)
if
the map
\begin{align}
    K _{ 0 }
    \times
    \bZ
    \to
    K,
    \qquad
    \left(
        \veca,
        m
    \right)
    \mapsto
    \veca
    +
    m
    \vecc
\end{align}
is a bijection.

\begin{theorem}\label{th:Qcoh Y}
    If
    \(
       K _{ 0 }
    \)
    is a fundamental domain of the translation by
    \(
       \vecc
    \),
    then one has an equivalence
\begin{align}\label{eq:Qcoh X}
  \Qcoh \bfX_K \simeq \Qcoh \cA_{\bfX_K}
\end{align}
of categories.
\end{theorem}


Let us summarize what we have done so far.
For any subset $K \subset L$
containing $\bZ \vecc$
and invariant under translation by $\vecc$,
we have constructed
a noncommutative variety
$
\bfX_K
$
and a sheaf of $\cO_{\bP^1}$-algebras
$
\cA_{\bfX_K}
$
on $\bP^1$
satisfying
\pref{eq:Qcoh X}.
On the other hand,
if $\lambda_i \ne \lambda_j$ for all $i \ne j$,
then one has $D^b \coh X_{\bsr,\bslambda}$
is equivalent to $D^b \module \Gamma$
since the path algebra of $\Gamma$ is isomorphic
to the endomorphism algebra of the full strong exceptional collection
\pref{eq:Geigle-Lenzing FSEC}.
Now we ask if there is a choice of $K$
such that the derived equivalence
holds for any $\bslambda$.
The answer is affirmative.
Take
\begin{align}\label{eq:I}
  I = \left\{ m \vecx_i \relmid m \in \bZ \text{ and }  i \in \{1, \dots, n \} \right\}
\end{align}
and write
\begin{align}\label{equation:bfY}
    \bfY \coloneqq \bfX_I
\end{align}
(or
$
\bfY_{\bsr,\bslambda} \coloneqq \left( \bfX_{\bsr,\bslambda} \right)_I
$
when we want to make the dependence on $\bsr$ and $\bslambda$ explicit).
This choice is motivated
by the fact that
for any $k \in \bZ$,
the set
$
  \left\{
    \veca \in I
  \relmid
    k \vecc \preceq \veca \preceq (k+1) \vecc
  \right\}
$
can naturally be identified with the set of vertices
of the quiver in~\pref{fg:orbifold projective line},
and we have the following:

\begin{theorem}\label{th:derived equivalence}
For any
$
\cI \in \Gr_{n-2} \left( \bfe_1 \bfk Q \bfe_0 \right),
$
one has a quasi-equivalence
\begin{align}\label{eq:derived equivalence}
  D^b \coh \bfY \simeq D^b \module \Gamma
\end{align}
of dg categories,
where $\Gamma$ is defined in~\pref{eq:Gamma}.
\end{theorem}

Therefore,
there are two `geometric' interpretations
of the points on the boundary of the compact moduli space
of orbifold projective lines.
One is `commutative' and given by the stack $\bfX_{\bsr,\bslambda}$,
and the other is `noncommutative' and given by the noncommutative variety
$\bfY_{\bsr,\bslambda}$.
Both are closely related to the quiver $\Gamma$ with relations.
The former is the fine moduli stack of refined representations of $\Gamma$,
and the latter is derived equivalent to $\Gamma$.
Moreover,
any choice of a subset $K$ of the Picard group $L$
containing $\bZ \vecc$
and invariant under translation by $\vecc$
produces a noncommutative variety $(\bfX_{\bsr, \bslambda})_K$
with an equivalence \pref{eq:Qcoh X},
and $\bfX_{\bsr,\bslambda}$
and $\bfY_{\bsr,\bslambda}$
are given by $K=L$ and $K=I$ respectively.

Now we ask if the constructions
of $\bfX_{\bsr,\bslambda}$
and $\bfY_{\bsr,\bslambda}$
can be generalized
to curves which are not necessarily rational.
The right hand side of~\pref{eq:fiber product} is the fiber product
in the category of stacks,
and makes sense in much greater generality.
With the fact that
the orbifold projective line
$
\bfX_{\bsr,\bslambda}
$
is obtained from $\bP^1$
by the root construction
of orders $\bsr$
at the points $\bslambda$ in mind,
let $C$ be a smooth curve and
$\scrX_r$ be the stack
obtained from
$
C \times C
$
by the $r$-th root construction
along the diagonal.
For $i \in \{ 1, 2 \}$,
let $\pi_{i} \colon \scrX_r \to C$
be the structure morphism
$
\scrX_r \to C \times C
$
of the root construction
composed with the $i$-th projection.
Now
\begin{align}\label{eq:general scrX}
  \pi_\bsr \colon
  \scrX_\bsr
  \coloneqq \scrX_{r_1}
  \ \!_{\pi_1}{\times}_{\pi_1} \scrX_{r_2}
  \ \!_{\pi_1}{\times}_{\pi_1} \cdots
  \ \!_{\pi_1}{\times}_{\pi_1} \scrX_{r_n}
  \xto{\pi_2 \times \cdots \times \pi_2}
  C^n
\end{align}
gives
a flat (but non-smooth) family of proper orbifold curves,
which restricts to
a smooth family
\begin{align}\label{eq:scrX}
  \pi_\bsr|_{\pi_\bsr^{-1}(C^n \setminus \Delta)} \colon
    \pi_\bsr^{-1}(C^n \setminus \Delta)
    \to
    C^n \setminus \Delta
\end{align}
over the complement of the big diagonal.
The fiber of $\pi_\bsr$ over $\bslambda \in C^n$
will be denoted by $\bfX_{\bsr,\bslambda}$, generalizing the case where
\(
    C = \bP^1
\).

The construction of $\cA_\bfY$ also makes sense
in this generality.
Set
\begin{align} \label{eq:scrY_bsr}
  \cA_{\scrY_\bsr} \coloneqq \varpi_{ \ast } \cEnd \left( \bigoplus_{\bsa \in I_0} \cO_{\scrX_\bsr}(\bsa) \right)
\end{align}
where
$
\varpi \colon \scrX_\bsr \to C \times C^n
$
is the structure morphism to the coarse moduli scheme,
\begin{align}
  I_0 \coloneqq \bigcup_{k=1}^n \left\{ (a_i)_{i=1}^n \in \bZ^n \relmid a_i = 0 \text{ for } i \ne k
  \text{ and } 0 \le a_k < r_k \right\},
\end{align}
and
\begin{align}
  \cO_{\scrX_\bsr}(\bsa) \coloneqq \cO_{\scrX_{r_1}}(a_1) \boxtimes
  \cdots \boxtimes \cO_{\scrX_{r_n}}(a_n)
\end{align}
is the exterior tensor product
of the $a_i$-th tensor powers
of the universal bundles
on the root stacks $\scrX_{r_i}$.
The symbol $\scrY_\bsr$ alone does not make sense,
and is meant to denote the family of noncommutative algebraic varieties
over $C^n$ such that
\(
   \Qcoh \scrY _{ \bsr }
   =
    \Qcoh \cA _{ \scrY _{ \bsr } }
\).
For $\bslambda \in C^n$,
we write the restriction of $\cA_{\scrY_\bsr}$
to the fiber $C \times \bslambda$ over $\bslambda$
as $\cA_{\bfY_{\bsr,\bslambda}}$.
It is clear from \pref{eq:A_XK} and \pref{eq:scrY_bsr}
that this $\cA_{\bfY_{\bsr,\bslambda}}$
is a generalization of that
in \eqref{equation:bfY}
to the case where $C$ may not be $\bP^1$.

\begin{theorem}\label{th:main}
For any smooth curve $C$,
any positive integer $n$,
and
any $\bsr \in \left( \bZ^{> 1} \right)^n$,
one has the following:
\begin{enumerate}[1.]
\item\label{it:flat}
The sheaf
$\cA_{\scrY_{\bsr}}$
of
$\cO_{C \times C^n}$-algebras
is flat over $C^n$.
\item\label{it:equivalent}
For any $\bslambda \in C^n \setminus \Delta$,
one has an equivalence
$
\Qcoh \cA_{\bfY_{\bsr,\bslambda}}
\simeq
\Qcoh \bfX_{\bsr,\bslambda}.
$
\item\label{it:smooth}
For any $\bslambda \in C^n$,
the category $\Qcoh \cA_{\bfY_{\bsr,\bslambda}}$
has finite homological dimension.
\end{enumerate}
\end{theorem}

\pref{th:main} can be compared with the existence
of a `degenerate' noncommutative cubic surface
of finite homological dimension
whose commutative counterpart is a singular cubic surface.
Unlike a noncommutative crepant resolution of a quotient singularity,
the stack $\bfX_{\bsr,\bslambda}$ is singular
whereas the `noncommutative contraction' $\bfY_{\bsr,\bslambda}$,
which is `below' $\bfX_{\bsr,\bslambda}$, is smooth.

The existence of distinct flat extensions $\Qcoh \scrX_\bsr$ and $\Qcoh \scrY_\bsr$
of the family~\pref{eq:scrX}
of abelian categories over $C^n \setminus \Delta$
illustrates the non-separatedness
of the `moduli space of abelian categories'.
When the coarse moduli space of $\bfX$ is $\bP^1$,
these extensions come from
two subsets of $L \simeq \Pic \bfX$.
While these subsets make sense
only when the coarse moduli space of $\bfX$ is $\bP^1$,
there is a third choice which makes sense
for any curve, namely, $\bZ \vecomega \subset L$
for the dualizing element $\vecomega$.
This choice is canonical,
and leads to a modular compactification
of the moduli stack of orbifold curves.
This is a `purely commutative' story,
which is discussed in a separate paper~\cite{ACOU1}.

This paper is organized as follows:
In~\pref{sc:stack},
we prove~\pref{lm:fiber_product},
of which~\pref{th:fiber product} is an immediate consequence.
In~\pref{sc:affine line},
we discuss the sheaves of algebras
$\cA_\bfX$ and $\cA_\bfY$
in the case when $C$ is an affine line
to illustrate their definitions
and to motivate the introduction
of the path algebra of a quiver
labeled by effective divisors
in~\pref{sc:path algebras},
which can describe $\cA_\bfY$
as shown in~\pref{sc:A_Y quiver}.
Theorems~\ref{th:main},~\ref{th:Qcoh Y},~and~\ref{th:derived equivalence}
are shown in Sections~\ref{sc:hdOQ},~\ref{sc:I-algebra},~and~\ref{sc:FSEC}
respectively.

%
%
\subsection{Notations and Conventions}

We will work over an algebraically closed field
$\bfk$
of characteristic zero
throughout the paper.
In particular,
all schemes, stacks, and their isomorphisms are defined over $\bfk$,
and all (dg) categories and functors are linear over $\bfk$.

\subsection*{Acknowledgements}
We thank the anonymous referees for reading the paper carefully
and suggesting many improvements.
During the preparation of this work,
D.~C.~was partially supported by the Australian Research Council Discovery Project grant DP220102861, 
S.~O.~was partially supported by
JSPS Grants-in-Aid for Scientific Research
(16H05994,
16H02141,
16H06337,
18H01120,
19KK0348,
20H01797,
20H01794,
21H04994),
and
K.~U.~was partially supported by
JSPS Grants-in-Aid for Scientific Research
(21K18575).

%
%
\section{Fiber products of global quotients}\label{sc:stack}

\begin{lemma}\label{lm:fiber_product}
Let 
$
\bY_i = \left[ X_i \middle/ G_i \right]
$
for $i=0,\ldots,n$
be the stack quotients of schemes $X_i$
by actions of algebraic groups $G_i$.
Let further
$
X_i \to X_0
$
for $i=1,\ldots,n$
be morphisms of schemes
that are intertwined by group homomorphisms
$
G_i \to G_0
$
yielding morphisms of stacks
$
\bY_i \to \bY_0.
$
Then one has an isomorphism
\begin{align}\label{equation:fiber product is a quotient stack}
  \bY_1 \times_{\bY_0} \cdots \times_{\bY_0} \bY_n
  \simto \bY \coloneqq \left[ X \middle/ G \right]
\end{align}
of stacks
where
$
X \coloneqq X_1 \times_{X_0} \cdots \times_{X_0} X_n
$
and
$
G \coloneqq G_1 \times_{G_0} \cdots \times_{G_0} G_n.
$
\end{lemma}

See,
e.g.,
\cite[\href{https://stacks.math.columbia.edu/tag/003O}{Tag 003O}]{stacks-project} for the definition of fiber products.

\begin{proof}
    We begin with defining the morphism~\eqref{equation:fiber product is a quotient stack} as a functor between categories fibered in groupoids over the category of
    \(
       \bfk
    \)-schemes.

    To define the action of the functor on objects, fix a test scheme $S$.
    An object of the left hand side of~\eqref{equation:fiber product is a quotient stack} over
    \(
       S
    \)
    is an
    \(
       n
    \)-tuple whose
    \(
       i
    \)-th entry is as follows:

    \begin{itemize}
        \item 
        For
        \(
           i = 1, \dots, n
        \)
        a principal
        \(
           G _{ i }
        \)-bundle
        \begin{align}
            P _{ i } \to S
        \end{align}
        
        \item
        For
        \(
           i = 1, \dots, n
        \)
        a
        \(
           G _{ i }
        \)-equivariant morphism        
        \begin{align}
            f _{ i } \colon P _{ i } \to X _{ i }
        \end{align}

        \item
        An isomorphism of principal
        \(
           G _{ 0 }
        \)-bundles
        \begin{align}
            \varphi _{ i j }
            \colon
            P _{ j }
            \times
            ^{
                G _{ j }
            }
            G _{ 0 }
            \simto
            P _{ i }
            \times
            ^{
                G _{ i }
            }
            G _{ 0 }
        \end{align}
        for
        \(
           i, j = 1, \dots, n
        \)
        satisfying the following conditions.
        \begin{align}\label{equation:conditions for varphi}
            \varphi _{ i i }
            =
            \id
            \\
            \varphi
            _{
                i j
            }
            \varphi
            _{
                j k
            }
            \varphi
            _{
                k i
            }
            =
            \id
            \\
            f _{ i }
            \varphi
            _{
                i j
            }
            =
            f _{ j }
        \end{align}
    \end{itemize}
    
    Let
    \(
       e _{ 0 }
       \in
       G _{ 0 } ( \bfk )
    \)
    be the identity. Take
    \(
       P
    \)
    to be the limit of the diagram consisting of
    \(
       \varphi
       _{
        i j
       }
    \)
    for
    \(
       i, j
       =
       1, \dots, n
    \)
    and
    \(
       P
       _{
        i
       }
       \xrightarrow{
        \id _{ P _{ i } }
        \times
        e _{ 0 }
       }
       P _{ i }
       \times
       ^{
           G _{ i }
       }
       G _{ 0 }
    \)
    for
    \(
       i = 1, \dots, n
    \).
    Then
    \(
       P
    \)
    admits a standard structure of a principal
    \(
       G
    \)-bundle over
    \(
       S
    \)
    together with a
    \(
       G
    \)-equivariant morphism
    \begin{align}
        f
        \colon
        P
        \to
        X.
    \end{align}
    Thus we have obtained an object on the right hand side of~\eqref{equation:fiber product is a quotient stack} over
    \(
       S
    \).

    We now move on to morphisms. A morphism
    \begin{align}\label{equation:morphism}
        \left(
            P _{ i },
            f _{ i },
            \varphi _{ i j }
           \right)
           \to
           \left(
            P ' _{ i },
            f ' _{ i },
            \varphi ' _{ i j }
           \right)       
    \end{align}
    in the fiber category of the left hand side of~\eqref{equation:fiber product is a quotient stack}
    over
    \(
       S
    \)    
    is an $n$-tuple of $G_i$-bundle isomorphisms
    \(
       \left(
        \psi _{ i }
        \colon
        P _{ i }
        \to
        P ' _{ i }
       \right)
       _{
        i = 1, \dots, n
       }
    \)
    such that
    \(
       f ' _{ i } \psi _{ i }
       =
       f _{ i }
    \).

    If we let
    \(
       f \colon P \to X
    \)
    and
    \(
       f ' \colon P ' \to X
    \)
    be the images of the source and the target of the morphism~\eqref{equation:morphism} under~\eqref{equation:fiber product is a quotient stack}, then the morphisms
    \(
       \psi _{ i }
    \)
    induce an isomorphism of principal
    \(
       G
    \)-bundles
    \(
       \psi \colon P \to P '
    \)
    satisfying
    \(
       f ' \psi = f
    \).

    We now define a morphism which we will show is the inverse of~\eqref{equation:fiber product is a quotient stack}, based on that the left hand side of~\eqref{equation:fiber product is a quotient stack} is the limit of a diagram of stacks.

    In order to define the morphism
    \(
        p _{ i }
        \colon
       \bY
       \to
       \bY _{ i }
    \)
    for
    \(
       i = 1, \dots, n
    \),
    take an object of $\bY$ over $S$; i.e., a $G$-bundle
    \(
       P \to S
    \)
    together with an $G$-equivariant morphism
    \(
       f
       \colon
       P \to X
    \).
    To such data, via the projection
    \(
       G \to G _{ i }
    \),
    we associate the
    \(
       G _{ i }
    \)-bundle
    \(
       P _{ i }
       \coloneqq
       P
       \times
       ^{
        G
       }
       G _{ i }
    \)
    and the
    \(
       G _{ i }
    \)-equivariant morphism defined as follows:
    \begin{align}
        f _{ i }
        \colon
        P _{ i }
        \xrightarrow{
         f
         \times
         \id
         _{
             G _{ i }
         }
        }
        X
        \times
        ^{
         G
        }
        G _{ i }
        \to
        X _{ i }
    \end{align}
    Thus we have obtained the map $\bY(S) \to \bY_i(S)$. It is straightforward to confirm that morphisms of the category
    \(
       \bY
    \)
    naturally induce those of
    \(
       \bY _{ i }
    \), so we omit the details.

    We also have to construct for
    \(
       1
       \le
       i \neq j
       \le
       n
    \)
    a natural isomorphism
    \(
       q _{ i } p _{ i }
       \Rightarrow
       q _{ j } p _{ j }
    \).
    This is obtained from the canonical isomorphism
    \begin{align}
        P _{ i }
        \times
        ^{
            G _{ i }
        }
        G _{ 0 }
        \simto
        P _{ j }
        \times
        ^{
            G _{ j }
        }
        G _{ 0 }.
    \end{align}

    It remains to show that the morphism of stacks we have just constructed is inverse to~\eqref{equation:fiber product is a quotient stack}, which is straightforward and hence left to the reader.
\end{proof}

\begin{proof}[Proof of~\pref{th:fiber product}]
For $i \in \{ 1, \ldots, n \}$
one has
$
\bfX_{r_i,\lambda_i}
\simeq
\bP
\left(
    1,
    r _{ i }
\right)
\simeq
\left[ \left( \bA^2 \setminus \bszero \right) \middle/ \Gm \right]
$
where $\Gm$ acts on $\bA^2$ by
$
\Gm \ni \alpha \colon (x,y) \mapsto (\alpha x, \alpha^{r_i} y).
$
The morphism
$
\bfX_{r_i,\lambda_i} \to \bP^1
$
is given by
$
(x,y) \mapsto (y, x^{r_1})
$
if $i=1$
and
$
(x,y) \mapsto (x^{r_i} + \lambda_i y, y)
$
if $i \in \{ 2, \ldots, n \}$.
Now we apply~\pref{lm:fiber_product}
to $X_i = \bA^2 \setminus \bszero$
and $G_i = \Gm$ for $i \in \{ 0, \ldots, n \}$.
The group $L^\dual$ is isomorphic
to the fiber product of the morphisms
$\Gm \to \Gm$, $\alpha \mapsto \alpha^{r_i}$
for $i \in \{ 1, \ldots, n \}$.
It remains to show
\begin{equation}\label{eq:fib_prod1}
  \left( \bA^2 \setminus \bszero \right) \times _{ \bA^2 \setminus \bszero } \cdots
  \times _{ \bA^2 \setminus \bszero } \left( \bA^2 \setminus \bszero \right)
  \simeq
  \Spec S \setminus \bszero,
\end{equation}
where $S$ is defined in~\pref{eq:S}.
\pref{eq:fib_prod1} follows from
\begin{equation}\label{eq:fib_prod2}
  \bA^2  \times _{ \bA^2 } \cdots
  \times _{ \bA^2 } \bA^2
  \simeq
  \Spec S.
\end{equation}
The left hand side of~\pref{eq:fib_prod2}
is the spectrum of
\begin{align}
  \bfk[x_1,y_1] \otimes_{\bfk[x_0, y_0]} \cdots \otimes_{\bfk[x_0, y_0]} \bfk[x_n,y_n],
\end{align}
which is isomorphic to the quotient of
$
\bfk[x_1,y_1,\ldots,x_n,y_n]
$
by the ideal
\begin{align}
(x_1^{r_1} - y_2)
+(y_1 - x_2^{r_2})
+\left( y _{ 1 } - (x_i^{r_i}+\lambda_i y_i),
x_1^{r_1} - y_i \right)_{i=3}^n.
\end{align}
Note that this ideal also contains
$
y_i - x_2^{r_2}
$
for $i \in \{ 3, \ldots, n \}$.
By eliminating all the $y_i$'s,
we see that this ring is isomorphic to $S$,
and~\pref{th:fiber product} is proved.
\end{proof}


%
%
\section{Sheaves of algebras on the affine line}\label{sc:affine line}

In this section,
we describe the sheaves of algebras
$\cA_\bfX$ and $\cA_\bfY$
for the toy model
\begin{align}
  C
  = \bA^1_z
  \coloneqq \Spec \bfk[z]
\end{align}
to illustrate their definitions given in~\pref{sc:introduction}
and to motivate the construction in~\pref{sc:path algebras}.
Choose
\(
   r
   \in
   \bZ
   ^{
    > 1
   }
\)
and
\(
   \lambda
   \in
   \bfk
\).
Consider the group action
\(
   \bmu _{ r } \curvearrowright \bA ^{ 1 } _{ x }
\)
defined by
\begin{align}\label{equation:group action}
  \bmu_r \ni \zeta \colon x \mapsto \zeta (x-\lambda) + \lambda
\end{align}
and let
\begin{align}
  \bfX_{r,\lambda} = \left[ \Spec \bfk[x] \middle/ \bmu_r \right],
  \\
  \cA_\bfX = \cA_\bfY = \bfk[x] \rtimes \bmu_r\label{equation:cAbfX}.
\end{align}
The right hand side of~\eqref{equation:cAbfX} is the skew group ring with respect to the action~\eqref{equation:group action}.
The coarse moduli of
\(
    \bfX_{r,\lambda}
\)
is as follows, which at the same time depicts
\(
    \bfX_{r,\lambda}
\)
as the root stack of
\(
    \Spec \bfk[z]
\)
ramified of order
\(
   r
\)
at the origin.
\begin{align}
  \varpi \colon \bfX_{r,\lambda} \to \Spec \bfk[z],
  \quad
  z \mapsto (x-\lambda)^r
\end{align}

Our assumption that $\bfk$ is an algebraically closed field
of characteristic zero implies the isomorphism
$
    \bfk \left[ \bmu_r \right]
    \simeq
    \bfk^{\times r}
$
as $\bfk$-algebras and
\begin{align}
  \bfk[x] \rtimes \bmu _{ r }
  \simeq
  \bfk[x]^{\oplus r}
\end{align}
as $\bfk[x]$-modules.

One has
$
\bfk[x] \simeq \bfk[z]^{\oplus r}
$
as $\bfk[z]$-modules,
so that
\begin{align}
  \bfk[x] \rtimes \bmu _{ r }
  \simeq
  \bfk[z]^{\oplus r^2}
\end{align}
as $\bfk[z]$-modules.
If we write the universal bundle on
\(
    \bfX_{r,\lambda}
\)
as the root stack as $\cO(1)$,
then one has
\begin{align}
  \bfk[x] \rtimes \bmu _{ r }
  \simeq
  \varpi_{ \ast } \cEnd \left( \bigoplus_{a=0}^{r-1} \cO(a) \right).
\end{align}
The $\bfk$-algebra
$
\bfk[x] \rtimes \bmu _{ r }
$
is described by a quiver such that vertices are the idempotents $\bfe_i \in \bfk \left[ \bmu_r \right]$
for $i \in \Hom \left( \bmu_r, \Gm \right) \simeq \bZ/r\bZ$
and there is one arrow from $\bfe_i$ to $\bfe_{i+1}$
where $1 \coloneqq \deg (x-\lambda) \in \Hom \left( \bmu_r, \Gm \right)$.

For
\(
   n
   \ge
   1
\)
the $\bfk$-algebra
\begin{align}
  \cA_\bfX \simeq \cA_{r_1,\lambda_1} \otimes_{\bfk[z]} 
  \cdots \otimes_{\bfk[z]} \cA_{r_n,\lambda_n}
\end{align}
is described by a quiver with vertices $\bfe_{\bsi}$
for
$
\bsi \in \prod_{k=1}^n \bZ/r_k \bZ
$
and arrows
$
a_{\bsi,j}
$
for $(\bsi,j) \in \prod_{k=1}^n \bZ/r_k \bZ \times \{ 1, \ldots, n \}$,
equipped with relations.
The arrow $a_{\bsi,j}$ goes from $\bfe_{\bsi}$ to $\bfe_{\bsi'}$
where $\bsi'$ is obtained from $\bsi$
by increasing the $j$-th component by $1$.
The algebra $\cA_\bfY$ is described
by the full subquiver consisting of vertices $\bfe_{\bsi}$
where $\bsi = (i_1,\ldots,i_n)$ runs over the subset of
$\prod_{k=1}^n \bZ / r_k \bZ$
such that $i_k$ is non-zero
for at most one $k \in \{ 1, \ldots, n \}$.
In order to describe $\cA_\bfY$ not as a $\bfk$-algebra
but as a $\bfk[z]$-algebra
(or a sheaf of $\cO_C$-modules),
we introduce the notion of
the \emph{path algebra of a $\Div_\eff C$-labeled quiver}
in~\pref{sc:path algebras} below.

%
%
\section{Path algebras of quivers with arrows labeled by effective divisors}\label{sc:path algebras}

A \emph{quiver}
$
Q = (Q_0,Q_1,s,t)
$
consists of a set $Q_0$ of \emph{vertices},
a set $Q_1$ of \emph{arrows},
and two maps $s,t \colon Q_1 \to Q_0$
called the \emph{source} and the \emph{target}.
The quiver $Q$ is said to be \emph{finite}
if the sets $Q_0$ and $Q_1$ are finite.
The \emph{path category} of the quiver
is the category with objects $Q_0$
freely generated by $Q_1$.
A \emph{path}
is a morphism of the path category.
The set of paths is denoted by $\cP$.

We consider only finite quivers unless otherwise specified.

\begin{definition}
  A \emph{labeling}
  of a quiver $Q = (Q_0,Q_1,s,t)$
  by a set $S$
  is a map
  \begin{align}
    D_{\bullet} \colon Q_1 \to S, \qquad
    \rho \mapsto D_\rho
  \end{align}
  of sets.
\end{definition}

If $S$ is a monoid,
then an $S$-labeling $D_\bullet$ induces
a functor from the path category of $Q$
to the category
with one object and the set $S$ of morphisms.
The resulting map
$
\cP \to S
$
of sets will be denoted by $D_\bullet$ again
by abuse of notation.

We henceforth
restrict ourselves
to labeling
by the monoid
$
G = \Div_\eff C
$
of effective Cartier divisors
on a scheme $C$.

\begin{definition}
  A \emph{cycle} is a non-identity endomorphism of an object
  of the path category, i.e.,
  a product
  $a_n \cdots a_2 a_1$
  of arrows
  $a_i \in Q_1$
  such that
  $s(a_{i+1}) = t(a_i)$
  for $i = 1, \ldots, n-1$
  and
  $s(a_1) = t(a_n)$.
  It is \emph{simple}
  if $s(a_1), s(a_2), \ldots, s(a_n)$
  are mutually distinct elements of $Q_0$.
\end{definition}

The set $\cP$ of paths is partitioned
into the disjoint union
of the set $\cP ^{ a }$ of acyclic (cycle-free) paths
and the set $\cP ^{ c }$ of paths
containing at least one cycle;
\begin{align}
\cP  = \cP ^{ a } \sqcup \cP ^{ c }.
\end{align}

Given a simple cycle $\rho$,
permuting the arrows cyclically naturally gives another simple cycle.

\begin{definition}\label{df:path algebra}
The \emph{path algebra} of a $\Div_\eff C$-labeled quiver
$
\cQ = (Q, D_{\bullet})
$
is the sheaf
\begin{align}\label{eq:path algebra}
  \cO_C \cQ
  \coloneqq
  \left. \left( \bigoplus_{\rho \in \cP } \cO _C(-D_{\rho}) \rho\right) \middle/ \cI \right.
\end{align}
of $\cO _C$-algebras,
where the multiplication is given by the concatenation of paths,
and the ideal $\cI$ is generated by the following relations:
For any vertex $v \in Q_0$
and any simple cycle $\rho$ at $v$,
we identify $\cO_C (-D_{\rho})\rho$
with $\cO_C (-D_{\rho})\bfe_v \subseteq \cO \bfe_v$,
where $\bfe_v$ is the trivial path at $v$.
\end{definition}

Since any simple cycle can be removed
by the relation $\cI$ in~\pref{eq:path algebra},
the inclusion and the projection gives a surjection
\begin{align}\label{eq:surjection to OQ}
  \bigoplus_{\rho \in \cP ^{ a }} \cO_C (-D_{\rho})\rho \to \cO_C \cQ
\end{align}
of $\cO_C$-modules.
The assumption that a quiver is finite implies that the set $\cP ^{ a }$
and hence the sheaf $\cO \cQ$ of $\cO_C$-algebras is also finite.

%
%
\section{The labeled quiver describing \texorpdfstring{$\cA_\bfY$}{AY}}\label{sc:A_Y quiver}

Fix a smooth curve $C$,
a positive integer $n$,
a sequence $\bsr = (r_i)_{i=1}^n$
of positive integers,
and a sequence $\bslambda = (\lambda_i)_{i=1}^n$
of points on $C$.
Let
$
\cA_\bfY
=
\cA_{\bfY_{\bsr,\bslambda}}
$
be the fiber
over $\bslambda \in C^n$
of the family
\pref{eq:scrY_bsr}
of sheaves of $\cO_C$-algebras.
Let further
$\cQ$
be the quiver obtained from the quiver in~\pref{fg:orbifold projective line}
by identifying the leftmost vertex $0$ with the rightmost vertex $1$,
equipped with the $\Div_\eff(C)$-labeling
defined by
\begin{align}
  D_{a_{ij}} =
  \begin{cases}
    \lambda_i & j = r_i, \\
    0 & \text{otherwise}.
  \end{cases}
\end{align}

Note that for each
\(
   1 \le i \le n
\)
and
\(
    1 \le j \le r _{ i }
\)
there is a standard isomorphism of line bundles
\begin{align}\label{eq:isomorphism}
    \cO _{ C }
    a _{
        i j
    }
    &
    \simto
    \varpi _{ \ast }
    \cO
    \left(
        \vecx _{ i }
    \right)
    \stackrel{
        \textrm{canonical}
    }{
        \simeq
    }
    \varpi _{ \ast }
    \cHom
    \left(
        \cO
        \left(
            ( j - 1 )
            \vecx _{ i }
        \right),
        \cO
        \left(
            j
            \vecx _{ i }
        \right)
    \right)
\end{align}
which sends
\(
   a _{ i j }
\)
to the tautological section of the line bundle
\(
   \cO
   \left(
    \vecx _{ i }
   \right)
\).
\pref{lm:A=OQ} below is an immediate consequence
of the definition of $\cO_C \cQ$ in~\pref{df:path algebra}.

\begin{lemma}\label{lm:A=OQ}
    The isomorphisms~\eqref{eq:isomorphism} induce the following isomorphism of $\cO_C$-algebras.
    \begin{align}
        \cO_C \cQ
        &
        \simto
        \cA_\bfY
    \end{align}
\end{lemma}

\begin{example}
To illustrate the above constructions,
we give an example of a family of algebras over
\(
   \bP ^{ 1 }
\)
parametrized by the points of
\(
   \bA ^{ 1 }
\).
Consider the labeled quiver $\cQ$ as in~\pref{eq:example of a labeled quiver},
where
$
D_1, D_2 \in \Div_\eff \left( \bP^1_{[u_0:u_1]} \times \bA^1_\lambda \right)
$
are defined by
$
s_1 \coloneqq u_0 - \lambda u_1,
s_2 \coloneqq u_0 - 2\lambda u_1
\in H^0 \left( \cO_{\bP^1} (1) \boxtimes \cO_{\bA^1} \right).
$
\begin{align}\label{eq:example of a labeled quiver}
  \begin{tikzpicture}
  [
    ->,
    thick,
    >=stealth,
    vertex/.style={circle,fill=white,draw},
    baseline=(current  bounding  box.west),
    node distance=2cm
  ]
    \node[vertex] (0) {0};
    \node[vertex] (1) [left=of 0] {1};
    \node[vertex] (2) [right=of 0] {2};
    \draw (0) to[bend left=30] node[pos=.5,fill=white] {$0$} (1);
    \draw (1) to[bend left=30] node[pos=.5,fill=white] {$D_1$} (0);
    \draw (0) to[bend left=30] node[pos=.5,fill=white] {$0$} (2);
    \draw (2) to[bend left=30] node[pos=.5,fill=white] {$D_2$} (0);
\end{tikzpicture}
\end{align}
The path algebra $\cO \cQ$
can be described as the matrix algebra 
\begin{align}
\begin{bmatrix}
\cO & \cO \left( - D _{ 1 } \right) & \cO \left( - D _{ 2 } \right)\\
\cO & \cO & \cO \left( - D _{ 2 } \right)\\
\cO & \cO \left( - D _{ 1 } \right) & \cO
\end{bmatrix},
\end{align}
which gives a flat family of finite $\bP^1_{[u_0:u_1]}$-algebras over $\bA^1_\lambda$.
If $\lambda \ne 0$,
then $D_1$ and $D_2$ give distinct points on $\bP^1$,
so that $\coh \cO \cQ$ is equivalent
to the category of coherent sheaves on a smooth rational orbifold projective curve
with two stacky points of order 2.
If $\lambda = 0$, then both $D_1$ and $D_2$ gives the point
$
p = [0:1] \in \bP^1_{[u_0:u_1]}.
$
The resulting path algebra
\begin{align}
\begin{bmatrix}
\cO & \cO(-p) & \cO(-p)\\
\cO & \cO & \cO(-p)\\
\cO & \cO(-p) & \cO
\end{bmatrix}
\end{align}
has a right module $\bfe_0 \cO_p$
of homological dimension two by~\pref{pr:hdOQ local}.
\end{example}

%
%
\section{Homological dimension of the path algebra}\label{sc:hdOQ}

\begin{definition}
  A quiver is said to \emph{have transverse cycles}
  if any pair of simple cycles intersect
  either at most one vertex
  or are cyclic permutations of each other.
\end{definition}

The following result is well-known. Lacking a good reference, we provide the short proof. 

\begin{proposition}\label{pr:OQbasis}
    If
    \(
        \cQ = (Q, D_{\bullet})
    \)
    is a $\Div_\eff C$-labeled quiver as in~\pref{df:path algebra} such that $Q$ has transverse cycles,
    then the morphism~\pref{eq:surjection to OQ}
    of $\cO_C$-modules is an isomorphism.
\end{proposition}

\begin{proof}
    First we study the map at the generic point of the curve
    \(
        C
    \).
    For this, let $\cK$ be the function field of $C$.
    The ideal $\cI_\cK \triangleleft \cK \cQ$ of relations at the generic point
    is generated by elements of the form
    $\rho - \bfe_v$ where $\rho$ ranges over all simple cycles and
    $v$ is the starting vertex of $\rho$.
    Given any path $\rho \in \cP ^{ c }$,
    we let $\bar{\rho}$ be a cycle-free path
    obtained from $\rho$ by contracting all simple cycles
    so that $\rho \equiv \bar{\rho}$ modulo $\cI_\cK$.
    Furthermore,
    our assumption that
    $Q$ has transverse cycles ensures that
    $\bar{\rho}$ is uniquely determined by $\rho$.
    In fact,
    to see the elements $\rho - \bar{\rho}$ form a $\cK$-basis for $\cI_\cK$,
    we use Bergman's diamond lemma  as follows (see \cite[Theorem~1.2]{Bergman} for its statement and associated terminology. Generalization to path algebras has been settled in, say,~\cite[Section~2.2]{Crawley-Boevey-ncag}).  We view $\cK \cQ$ as an algebra over the semi-simple algebra $\cK^{Q_0}$ freely generated by the bimodule $\cK Q_1$. We partially order monomials in edges by their length. The reduction system we use replaces each monomial of edges $a_1 a_2 \ldots a_l$ corresponding to a simple cycle $\rho$ beginning at $v$ with $\bfe_v$. Furthermore, given $a \in Q_1$, it replaces $\bfe_v a$ and $a \bfe_w$ with $a$ whenever $v$ is the target of $a$ and $w$ is its source, and to $0$ otherwise. We need to check the overlap $m = a_i \ldots a_l a_1 \ldots a_l$ which can be reduced using $\rho$ to give $a_i \ldots a_l \bfe_v = a_i \ldots a_l$. If $w$ is the source of $a_i$, then the monomial $m$ can also be reduced using the cycle $a_i \ldots a_l a_1 \ldots a_{i-1}$ to $\bfe_w a_i \ldots a_l = a_i \ldots a_l$. The two reductions are the same so the overlap ambiguity is resolved. Now the transverse cycle condition ensures that the only other overlaps to check involve idempotents $\bfe_v$, and these are easily resolved. 
    This proves~\pref{pr:OQbasis}
    at the generic point. As the source of the map~\pref{eq:surjection to OQ} is torsion free, this already implies the assertion.
\end{proof}

Given a simple cycle $\rho$ on a quiver $Q$,
we can construct a new quiver $Q'$
by \emph{contracting the cycle $\rho$}
as follows.
The set of vertices is defined by $Q'_0=Q_0/\!\sim$,
where $\sim$ is the equivalence relation
which identifies all vertices of $\rho$ but leaves all other vertices distinct.
The set of arrows $Q'_1$ is the subset of $Q_1$
consisting of arrows not contained in $\rho$.
A labeling
$
D_\bullet \colon Q_1 \to S
$
restricts to a labeling
$
D'_\bullet \coloneqq D_\bullet|_{Q_1'},
$
and we write $\cQ' = (Q',D'_\bullet)$
for $\cQ = (Q, D_\bullet)$.

\begin{proposition}\label{pr:Morita}
If $\cQ$ is a $\Div_\eff C$-labeled quiver
having transverse cycles
and
$\cQ'$ is the $\Div_\eff C$-labeled quiver
obtained by contracting a simple cycle $\rho$
with $D_\rho = 0$,
then one has an equivalence
\begin{align}
  \Qcoh \cO_C\cQ \simeq \Qcoh \cO_C\cQ'
\end{align}
of categories.
\end{proposition}

\begin{proof}
    Let $v,w \in Q_0$ and
    consider the local projectives $P_v = \bfe_v \cO_C\cQ$
    and $P_w = \bfe_w \cO_C\cQ$.
    If $v, w$ are in $\rho$, as
    \(
       D _{
        \rho
       }
       =
       0
    \),
    multiplication by the ``arc'' of $\rho$ from the vertex
    \(
       w
    \)
    to
    \(
       v
    \)
    induces an isomorphism
    $
        P_w \simto P_v.
    $
    Thus the direct sum
    $P = \oplus_v P_v$
    over all vertices of $Q$ outside $\rho$ and one vertex of $\rho$
    gives a local progenerator of $\cO_C \cQ$. 

    It thus suffices to show that
    $\cO_C \cQ' \simeq \cEnd_C P$.
    From~\pref{pr:OQbasis},
    we know that for each pair
    \(
       \left(
        v, w
       \right)
    \)
    of vertices of
    \(
       Q
    \)
    \begin{align}
        \cHom_{\cO_C \cQ}(P_w,P_v)
        = \bfe_v \cO_C\cQ \bfe_w
        = \bigoplus_{\gamma} \cO_C(-D_{\gamma}) \gamma,
    \end{align}
    where $\gamma$ runs over all cycle-free paths from $w$ to $v$.
    Since $Q$ has transverse cycles, the path of
    \(
       Q '
    \)
    which is obtained from an cycle-free path of
    \(
       Q
    \)
    in the obvious manner is again cycle-free.
    Moreover this yields a bijection between cycle-free paths from
    \(
       w
    \)
    to
    \(
       v
    \)
    and cycle-free paths in $Q'$
    from the image of $w$ to the image of $v$.
    The assumption $D_{\rho}=0$ ensures that this bijection respects the labeling,
    and~\pref{pr:Morita} is proved.
\end{proof}

\begin{definition}
A $\Div_\eff C$-labeling $D_\bullet$
of a quiver $\cQ$
is said to be \emph{reduced}
if
for every simple cycle $\rho$ on $Q$,
the divisor $D_{\rho}$ is reduced.
\end{definition}

\pref{lm:transverse1pt} below is an immediate consequence
of the definition of having transverse cycles.

\begin{lemma}\label{lm:transverse1pt}
If $Q$ has transverse cycles,
then any two simple cycles which intersect in more than one vertex are cyclic permutations of each other. 
\end{lemma}

The proof of \pref{lm:preservecyclecond} below is straightforward.

\begin{lemma}\label{lm:preservecyclecond}
Let $\rho$ be a simple cycle on $\cQ$ with $D_{\rho} = 0$
and
$\cQ'$ be the labeled quiver obtained by contracting $\rho$.
If $Q$ has transverse cycles,
then so does $Q'$.
If $D_\bullet$ is reduced,
then so is $D_\bullet'$.
\end{lemma}

The condition of transverse cycles has the following nice consequence.

\begin{lemma}\label{lm:edgenzd}
    For a $\Div_\eff C$-labeled quiver $\cQ$
    with transverse cycles
    and an arrow $a \in Q_1$,
    the left multiplication by $a$ map
    $
        \bfe_{s(a)} \cK \cQ
        \xrightarrow{
            a \cdot
        }
        \bfe_{t(a)} \cK \cQ
    $
    over the function field $\cK$ of $C$ is injective.
\end{lemma}

\begin{proof}
    Suppose to the contrary that
    $
        a \sum_i r_i \rho_i = 0,
    $
    where
    $
        r_i \in \cK^\times
    $
    and
    $\rho_i$ are distinct cycle-free paths
    ending at $s(a)$.
    If
    \(
        a \rho _{ i }
    \)
    are cycle-free for all
    \(
        i
    \),
    then the assertion is obvious.
    For the contrary, suppose that
    \(
        a \rho _{ i }
    \)
    contains a cycle for some
    \(
        i
    \).
    By the transversality assumption there exists a unique simple cycle
    \(
        \gamma
        =
        a \gamma '
    \)
    at the vertex
    \(
        t ( a )
    \).
    Let
    \(
        \rho ' _{ i }
    \)
    be the unique acyclic path such that
    \(
        \rho _{ i }
        =
        \gamma '
        \rho ' _{ i }
    \).
    Now
    \begin{align}\label{equation:a sum}
        0
        =
        a    
        \sum
        _{
            i
        }
        r _{ i }
        \rho _{ i }
        =
        \sum
        _{
            a \rho _{ i }
            \in
            \cP ^{ c }
        }
        r _{ i }
        a \rho _{ i }
        +
        \sum
        _{
            a \rho _{ i }
            \in
            \cP ^{ a }
        }
        r _{ i }
        a \rho _{ i }
        =
        \sum
        _{
            a \rho _{ i }
            \in
            \cP ^{ c }
        }
        r _{ i }
        \bfe
        _{
            t _{ a }
        }
        \rho ' _{ i }
        +
        \sum
        _{
            a \rho _{ i }
            \in
            \cP ^{ a }
        }
        r _{ i }
        a \rho _{ i }.
    \end{align}
    As
    \(
       \rho ' _{ i }
    \)
    for those
    \(
       i
    \)
    such that
    \(
       a \rho _{ i }
       \in
       \cP ^{ a }
    \)
    are all distinct and do not contain the edge
    \(
       a
    \),
    the equality~\eqref{equation:a sum} implies
    \(
       r _{ i } = 0
    \)
    for all
    \(
       i
    \).
\end{proof}

\pref{th:hdOQ} below is a consequence of~\pref{pr:hdOQ local}
and~\pref{pr:hdOQ local to global}

\begin{theorem}\label{th:hdOQ}
The path algebra
$\cO_C \cQ$
of
a reduced $\Div_\eff C$-labeled quiver
$\cQ$
on a smooth curve $C$
with transverse cycles
has homological dimension at most two.
\end{theorem}

\begin{proposition}\label{pr:hdOQ local}
    \pref{th:hdOQ} holds if $C$ is the spectrum of a discrete valuation ring.
\end{proposition}

\begin{proof}
    Let $\cO$ be a discrete valuation ring
    with the closed point $p$
    and
    a uniformizing parameter $t$.
    Each vertex $v \in Q_0$ gives rise
    to the indecomposable projective module
    $P_v = \bfe_v \cO\cQ$
    and the simple module
    $S_v = P_v/ \rad P_v$.
    It suffices to show that $\pd S_v \le 2$. 

    Suppose first that there are no cycles through $v$.
    Then~\pref{pr:OQbasis} and~\pref{lm:edgenzd} imply that
    \begin{align}
        0
        \to \bigoplus_{t(a) = v} P_{s(a)}
        \xto{\binom{-t}{a}}
        \bigoplus_{t(a)=v} P_{s(a)} \oplus P_v
        \xto{(a \ t)} P_v
        \to S_v
        \to 0
    \end{align}
    gives a projective resolution of
    $
        S_v = \bfe_v \cO/(t),
    $
    so that
    $
        \pd S_v \le 2
    $
    in this case.

    Suppose now that there exists a simple cycle $\rho$ starting at $v$.
    If $D_{\rho} = 0$,
    then we may contract $\rho$ to obtain a new $\Div_\eff C$-labeled quiver,
    whose path algebra is Morita equivalent to the original one by~\pref{pr:Morita},
    and which
    still satisfies the hypotheses of~\pref{th:hdOQ}
    by~\pref{lm:preservecyclecond}.
    By repeating this operation if necessary,
    we may assume that there exist exactly $r+1$ simple cycles
    $\rho_0, \ldots, \rho_r$
    starting at $v$
    and that $D_{\rho_i} = p$ for all $i$.
    \pref{lm:transverse1pt} implies that
    $\rho_i$ and $\rho_j$ for $i \ne j$ intersect only at the vertex $v$.
    Set
    \begin{align}
    \rho_i = a_i \rho'_i
    \end{align}
    and
    \begin{align}
        \Xi_v
        \coloneqq
        \left\{
            a \in Q_1 \relmid t(a) = v
        \right\}
        \setminus
        \left\{
            a_0, \ldots, a_r
        \right\}.
    \end{align}
    Since $D_{\rho_i} = p \ne 0$ for all $i$,
    we again have 
    $
        S_v
        =
        \bfe_v \cO/(t).
    $
    We claim that we have a projective resolution
    \begin{align}\label{eq:Sv}
    0
    \to
    \bigoplus
    _{
        i = 1
    }
    ^{
        r
    }
    P _{ v }
    \oplus
        \bigoplus
            _{
                a
                \in
                \Xi
                _{
                    v
                }
            }
            P _{
                s ( a )
            }
    \xto{\ \phi \ }
    \bigoplus_{t(a)=v} P_{s(a)}
    \xto{(a)} P_v
    \to S_v
    \to 0,
    \end{align}
    where $\phi$ is defined by the following block matrix (we think of the elements of the source and the target as column vectors):
    \begin{equation}
        \left[
            \begin{array}{c:c}
                \begin{matrix}
                    - \rho ' _{ 0 }
                    &
                    \cdots
                    &
                    - \rho ' _{ 0 }
                \end{matrix}
                &
                \begin{pmatrix}
                    \rho ' _{ 0 }
                    a
                \end{pmatrix}
                _{
                    a
                    \in
                    \Xi _{ v }
                }
                \\
                \hdashline
                \begin{matrix}
                    \rho ' _{ 1 }
                    &
                    &
                    \\
                    &
                    \ldots
                    &
                    \\
                    &
                    &
                    \rho ' _{ r }
                \end{matrix}
                &
                0
                \\
                \hdashline
                0
                &
                \left(
                    - t
                \right)
                _{
                    a
                    \in
                    \Xi _{ v }
                }
            \end{array}
        \right]
        \colon
        \begin{bmatrix}
            \bigoplus
            _{
                i = 1
            }
            ^{
                r
            }
            P _{ v }
            \\
            \\
            \bigoplus
            _{
                a
                \in
                \Xi
                _{
                    v
                }
            }
            P _{
                s ( a )
            }
        \end{bmatrix}
        \to
        \begin{bmatrix}
            P
            _{
                s
                \left(
                    a _{ 0 }
                \right)
            }
            \\
            \\
            \bigoplus
            _{
                i = 1
            }
            ^{
                r
            }
            P _{
                s
                \left(
                    a _{ i }
                \right)
            }
            \\
            \\
            \bigoplus
            _{
                a
                \in
                \Xi
                _{
                    v
                }
            }
            P _{
                s ( a )
            }
        \end{bmatrix}
    \end{equation}
    

    It is easy to check that~\pref{eq:Sv} is a complex and
    \(
    \phi
    \)
    is injective. What remains to check is the inclusion
    $
        \ker (a) \subseteq \image \phi.
    $
    Take an element of
    \(
       \ker ( a )
    \); namely, a column vector
    \begin{align}\label{equation:element of ker ( a )}
        \begin{bmatrix}
            \sigma _{ 0 }
            \\
            \sigma _{ 1 }
            \\
            \vdots
            \\
            \sigma _{ r }
            \\
            \left(
                \sigma _{ a }
            \right)
            _{
                a
                \in
                \Xi _{ v }
            }
        \end{bmatrix}
        \in
        \ker ( a )
    \end{align}
    such that
    \begin{align}\label{equation:ker ( a )}
        0
        =
        ( a )
        \begin{bmatrix}
            \sigma _{ 0 }
            \\
            \sigma _{ 1 }
            \\
            \vdots
            \\
            \sigma _{ r }
            \\
            \left(
                \sigma _{ a }
            \right)
            _{
                a
                \in
                \Xi _{ v }
            }
        \end{bmatrix}
        =
        a _{ 0 } \sigma _{ 0 }
        +
        \sum
        _{
            i = 1
        }
        ^{
            r
        }
        a _{ i } \sigma _{ i }
        +
        \sum
        _{
            a
            \in
            \Xi _{ v }
        }
        a \sigma _{ a }.
    \end{align}
    Without loss of generality, we assume all entries of~\eqref{equation:element of ker ( a )} are linear combinations of cycle-free paths.
    By a direct analysis one can confirm that it can be modified by an element of
    $
        \image \phi
    $
    so that all entries of~\eqref{equation:element of ker ( a )} but
    \(
       \sigma _{ 0 }
    \)
    are zero.    
    Exactness now follows from~\pref{lm:edgenzd}, which ensures that
    $
        \ker(a_0 \colon P_{s(a_0)} \to P_v) = 0,
    $
    and~\pref{pr:hdOQ local} is proved.
\end{proof}


\begin{proposition}\label{pr:hdOQ local to global}
Let $\cA$ be a coherent sheaf of algebras
on a smooth curve $C$ such that at any point
\(
   p \in C
\)
the homological dimension of the stalk
\(
   \cA _{ p }
\)
is at most two.
Then the homological dimension of $\cA$ is at most two.
\end{proposition}

\begin{proof}
    It suffices to show that $\Ext^3_{\cA}(\scrF,-) = 0$
    for every coherent $\cA$-module $\scrF$.
    If $\scrF$ has dimension 0 as a sheaf on $C$,
    then this immediately follows from our assumption as the local-to-global Ext spectral sequence degenerates for the obvious reason.
    
    Since the torsion part of
    \(
       \scrF
    \)
    as an
    \(
       \cO _{ C }
    \)-module is automatically an
    \(
       \cA
    \)-submodule, it remains to show the assertion under the assumption that $\scrF$ is torsion free as a sheaf on $C$.
    In this case, if we can show that locally at every closed point
    $\scrF$ has projective dimension at most one,
    then the assertion follows from the local-to-global Ext spectral sequence.
    To this end, suppose that $\cO$ is a discrete valuation ring with a uniformizing parameter $t$.
    Consider the following part of the long exact sequence of Ext groups.
    \begin{align}\label{equation:multiplication by t}
        \Ext^2_{\cA}(\scrF,-)
        \xto{\ t \ }
        \Ext^2_{\cA}(\scrF,-)
        \to
        \Ext^3_{\cA}(\scrF/t\scrF,-)
    \end{align}
    The last term is zero by what we have already confirmed, as $\scrF/t\scrF$ is torsion. Hence the first map of~\eqref{equation:multiplication by t} is always surjective. By the structure theorem for finitely generated modules over a discrete valuation ring or Nakayama's lemma, we see $\Ext^2_{\cA}(\scrF,-) = 0$. Thus we conclude the proof.
\end{proof}

\begin{proof}[Proof of~\pref{th:main}]
\pref{th:main}.\ref{it:smooth} is a special case of~\pref{th:hdOQ}.
\pref{th:main}.\ref{it:flat} is clear since
$\cA_{\scrY_\bsr}$ is $\cO_{C \times C^n}$-locally-free.
In order to prove \pref{th:main}.\ref{it:equivalent},
take Zariski open subsets $U_i$ of $C$
such that
$\lambda_i \in U_i$,
$\lambda_j \nin U_i$ for $j \ne i$,
and
$\bigcup_{i=1}^n U_i = C$,
which is possible 
since $\bslambda \in C^n \setminus \Delta$.
Then the restriction of $\cA_{\bfY_{\bsr, \bslambda}}$
to $U_i$ is Morita equivalent to
$\cO_{U_i} \cQ_i$
by \pref{pr:Morita},
where $\cQ_i$ is the labeled quiver
associated with $n = 1$,
$\bsr = r_i$
and
$\bslambda = \lambda_i$.
Recall from \pref{sc:A_Y quiver} that
the quiver $\cQ_i$ is introduced in such a way that
$
\Qcoh \cO_{U_i} \cQ_i \simeq \Qcoh \lb \bfX_{\bsr, \bslambda} \times_C U_i \rb
$.
These equivalences on $U_i$ glue together to give the desired equivalence
on $C$, and \pref{th:main}.\ref{it:equivalent} is proved.
\end{proof}

%
%
\section{Graded algebras and \texorpdfstring{$I$}{I}-algebras}\label{sc:I-algebra}

\begin{definition}[{\cite[(4.2.1)]{MR1304753}}]\label{df:ample autoequivalence}
  Let $\scrC$ be an abelian category
  and $\cA$ be an object of $\scrC$.
  An autoequivalence $s$ of $\scrC$ is \emph{ample}
  if
  \begin{itemize}
    \item[(a)]
    for every object $\cM$ of $\scrC$,
    there exist positive integers $l_1,\ldots,l_p$
    and an epimorphism
    $
    \bigoplus_{i=1}^p \cA(-l_i) \to \cM,
    $
    and
    \item[(b)]
    for every epimorphism
    $
    f \colon \cM \to \cN,
    $
    there exists an integer $n_0$
    such that for every $n \ge n_0$,
    the map
    $
    H^0(\cM(n)) \to H^0(\cN(n))
    $
    is surjective,
  \end{itemize}
  where
  $
  H^0(\cM) \coloneqq \Hom(\cA,\cM)
  $
  and
  $
  \cM(l) \coloneqq s^l(\cM)
  $
  for $l \in \bZ$.
\end{definition}

\pref{lm:ample} below is an immediate consequence of~\pref{df:ample autoequivalence}.

\begin{lemma}\label{lm:ample}
    If $\cA$ is a coherent sheaf of algebras on a polarized scheme $(C, \cL)$,
    then the tensor product by $\cL$ 
    gives an ample autoequivalence of $\left( \coh \cA, \cA \right)$.
\end{lemma}

\pref{th:Artin-Zhang} below is one of the main results of~\cite{MR1304753}.

\begin{theorem}[{\cite[Theorem 4.5.(1)]{MR1304753}}]\label{th:Artin-Zhang}
  Let $(\scrC,\cA,s)$ be a triple consisting of
  \begin{itemize}
    \item an abelian category $\scrC$,
    \item an object $\cA$ of $\scrC$, and
    \item an autoequivalence $s$ of $\scrC$
  \end{itemize}
  satisfying the following three conditions:
  \begin{enumerate}[1.]
    \item $\cA$ is noetherian.
    \item $A_0 \coloneqq H^0(\cA)$ is a right noetherian ring and
    $H^0(\cM)$ is a finite $A_0$-module for all $\cM$.
    \item $s$ is ample.
  \end{enumerate}
  Then the graded ring
  $
  A \coloneqq \bigoplus_{i=0}^\infty H^0(\cA(i))
  $
  is right noetherian satisfying $\chi_1$,
  and $(\scrC,\cA,s)$ is isomorphic to $(\qgr A, \pi(A), (1))$.
\end{theorem}

Given a set $J$,
an \emph{$J$-algebra}
is a category
whose set of objects is identified with $J$.
A $\bZ$-algebra is a generalization
of $\bZ$-graded algebra~\cite{MR1230966}.
A $\bZ$-algebra analog of~\pref{th:Artin-Zhang}
is given in~\cite[Theorem~2.4]{MR2122731}.





Let
$\cA = \cO_C \cQ$ be the path algebra
of a quiver
$\cQ=(Q,D_{\bullet})$
labeled by effective divisors
on an integral polarized scheme $(C,\cL)$.
The graded ring associated with the triple
$(\coh \cA, \cA, \cL \otimes(-))$
as in~\pref{th:Artin-Zhang}
will be denoted by $A$.

Define an $J$-algebra
for $J \coloneqq Q_0 \times \bZ$
by
\begin{align}
  B_{(v,m)(w,n)} \coloneqq H^0 \left( \bfe_v \cA \bfe_w \otimes_{\cO_C} \cL^{\otimes (n-m)} \right),
\end{align}
whose multiplication
is induced from that of $\cA$
through the morphism
\begin{align}
  \left( \bfe_u \cA \bfe_v \otimes_{\cO_C} \cL^{\otimes (m-l)} \right)
   \otimes_{\cO_C}
  \left( \bfe_v \cA \bfe_w \otimes_{\cO_C} \cL^{\otimes (n-m)} \right)
  \to
  \bfe_u \cA \bfe_w \otimes_{\cO_C} \cL^{\otimes (n-l)}.
\end{align}
One can collapse the $J$-structure
to a $\bZ$-structure
by
\begin{align}
  B'_{mn} \coloneqq \bigoplus_{v,w \in Q_0} B_{(v,m)(w,n)}
\end{align}
without changing the categories $\gr$ and $\qgr$;
\begin{align}
  \gr B \simeq \gr B', \qquad
  \qgr B \simeq \qgr B'.
\end{align}

As the index set
\(
   K = I
\)
satisfies the assumption of~\pref{th:Qcoh Y},
it follows that the $\bZ$-algebra $B'$ is related to $A$ by
$B'_{mn} = A_{n-m}$,
so that
\begin{align}
  \gr A \simeq \gr B', \qquad
  \qgr A \simeq \qgr B'
\end{align}

If $\cQ$ is the labeled quiver
introduced in the beginning of~\pref{sc:A_Y quiver},
then the resulting $J$-algebra $B$ coincides
with the category $\sfS_I$
appearing in~\eqref{eq:sfSK}, where the index set
\(
   K
\)
is taken to be the set
\(
   I
\)
defined in~\eqref{eq:I},
so that~\pref{th:Artin-Zhang} gives
\begin{align}
  \coh \cA_\bfY
  \simeq \qgr A
  \simeq \qgr B
  \simeq \qgr \sfS _{ I }.
\end{align}
This proves the equivalence~\pref{eq:Qcoh X} for $K=I$.
The equivalence for general $K$ is proved similarly
by using the $K$-algebra $\sfS = \sfS_K$,
and~\pref{th:Qcoh Y} is proved.

%
%
\section{A full strong exceptional collection on \texorpdfstring{$\bfY$}{Y}}\label{sc:FSEC}

In this section,
we always assume
\(
   C = \bP ^{ 1 }
\).

\begin{theorem}\label{th:FSEC}
  The sequence
  $
  \left( \cO_\bfY \left( \veca \right) \right)_{0 \preceq \veca \preceq \vecc}
  $
  of objects of $\coh \bfY \coloneqq \qgr \sfS_I$
  is a full strong exceptional collection of
  \(
     D ^{ b } \coh \bfY
  \)
  whose total morphism algebra is the path algebra of the quiver
  in~\pref{fg:orbifold projective line}
  with relations~\pref{eq:canonical relations}.    
\end{theorem}

\begin{proof}
Under the equivalence with $\coh \cA_\bfY$,
the object
$
\cO_\bfY \left( a \vecx_i \right)
$
for
$
0 < a < r_i
$
corresponds to the $\cA_\bfY$-module
$\bfe_{i,a} \cA_\bfY$,
and the objects
$\cO_\bfY(0)$
and
$\cO_\bfY(\vecc)$
correspond to
$\bfe_0 \cA_\bfY$
and
$\bfe_0 \cA_\bfY(1)$
respectively.
Since they are sheaves of projective $\cA_\bfY$-modules,
Ext-groups between them can be computed without taking further local projective resolutions as $\cA_\bfY$-modules.
Namely, we have
\begin{align}
    \Hom ^{ \bullet }
    _{
        \coh \cA _{ Y }
    }
    \left(
        \bfe _{
            \veca
        }
        \cA _{ \bfY },
        \bfe _{
            \vecb
        }
        \cA _{ \bfY }
    \right)
    &\simeq
    \bfR \Gamma
    \left(
        \bfY,
        \cHom
        ^{
            \bullet
        }
        _{
            \cA _{ \bfY }
        }
        \left(
            \bfe _{
                \veca
            }
            \cA _{ \bfY },
            \bfe _{
                \vecb
            }
            \cA _{ \bfY }            
        \right)
    \right)
    \\
    &\simeq
    \bfR \Gamma
    \left(
        \bfY,
        \cHom
        _{
            \cA _{ \bfY }
        }
        \left(
            \bfe _{
                \veca
            }
            \cA _{ \bfY },
            \bfe _{
                \vecb
            }
            \cA _{ \bfY }            
        \right)
    \right) \\  
    &\simeq
    \bfR \Gamma
    \left(
        \bfY,
        \cO
        _{
            \bfY
        }
        \left(
            \vecb
            -
            \veca
        \right)
    \right) \\
    &\simeq
    \bfR \Gamma
    \left(
        \bP ^{ 1 },
        \varpi _{ \ast }
        \cO
        _{
            \bfY
        }
        \left(
            \vecb
            -
            \veca
        \right)
    \right).
\end{align}
It then follows that
$
    \left(
        \cO_\bfY
        \left(
            \veca
        \right)
    \right)
    _{
        0 \preceq \veca \preceq \vecc
    }
$
is a strong exceptional collection
whose total morphism algebra is the path algebra of the quiver
in~\pref{fg:orbifold projective line}
with relations~\pref{eq:canonical relations}.
In order to show that the collection
is full,
one can use the exact sequences
\begin{align}
  0
  \to
  \cO_\bfY(0)
  \to
  \cO_\bfY \left( a \vecx_i \right) \oplus \cO_\bfY \left( \vecc \right)
  \to
  \cO_\bfY \left( a \vecx_i + \vecc \right)
  \to
  0
\end{align}
and their translates by $\bZ \vecc$
to show that $\cA_\bfY(i)$ for any $i \in \bZ$
is contained in the full triangulated subcategory
generated by the collection.
\end{proof}

\pref{th:derived equivalence} is an immediate consequence of~\pref{th:FSEC}.

\bibliographystyle{amsalpha}
\bibliography{bibs}

\def\cprime{$'$} \def\cprime{$'$}
\providecommand{\bysame}{\leavevmode\hbox to3em{\hrulefill}\thinspace}
\providecommand{\MR}{\relax\ifhmode\unskip\space\fi MR }
\providecommand{\MRhref}[2]{%
  \href{http://www.ams.org/mathscinet-getitem?mr=#1}{#2}
}
\providecommand{\href}[2]{#2}
\begin{thebibliography}{{Sta}18}

\bibitem[ACOU]{ACOU1}
Tarig Abdelgadir, Daniel Chan, Shinnosuke Okawa, and Kazushi Ueda, \emph{A compact moduli of orbifold projective curves}, in preparation.

\bibitem[AOUa]{Abdelgadir-Okawa-Ueda_nccubic}
Tarig Abdelgadir, Shinnosuke Okawa, and Kazushi Ueda, \emph{Compact moduli of non-commutative cubic surfaces}, in preparation.

\bibitem[AOUb]{1411.7770}
\bysame, \emph{Compact moduli of noncommutative projective planes}, arXiv:1411.7770.

\bibitem[AU15]{MR3274027}
Tarig Abdelgadir and Kazushi Ueda, \emph{Weighted {P}rojective {L}ines as {F}ine {M}oduli {S}paces of {Q}uiver {R}epresentations}, Comm. Algebra \textbf{43} (2015), no.~2, 636--649. \MR{3274027}

\bibitem[AZ94]{MR1304753}
M.~Artin and J.~J. Zhang, \emph{Noncommutative projective schemes}, Adv. Math. \textbf{109} (1994), no.~2, 228--287. \MR{MR1304753 (96a:14004)}

\bibitem[Ber78]{Bergman}
George~M. Bergman, \emph{The diamond lemma for ring theory}, Adv. in Math. \textbf{29} (1978), no.~2, 178--218. \MR{506890}

\bibitem[BP93]{MR1230966}
A.~I. Bondal and A.~E. Polishchuk, \emph{Homological properties of associative algebras: the method of helices}, Izv. Ross. Akad. Nauk Ser. Mat. \textbf{57} (1993), no.~2, 3--50. \MR{1230966 (94m:16011)}

\bibitem[CB]{Crawley-Boevey-ncag}
William Crawley-Boevey, \emph{Noncommutative algebra 1}, course notes available \href{https://www.math.uni-bielefeld.de/~wcrawley/19noncommalg1/NA1-final.pdf}{here}.

\bibitem[GL87]{MR915180}
Werner Geigle and Helmut Lenzing, \emph{A class of weighted projective curves arising in representation theory of finite-dimensional algebras}, Singularities, representation of algebras, and vector bundles ({L}ambrecht, 1985), Lecture Notes in Math., vol. 1273, Springer, Berlin, 1987, pp.~265--297. \MR{915180}

\bibitem[GM82]{MR658730}
I.~M. Gel$\prime$fand and R.~D. MacPherson, \emph{Geometry in {G}rassmannians and a generalization of the dilogarithm}, Adv. in Math. \textbf{44} (1982), no.~3, 279--312. \MR{658730}

\bibitem[MFK94]{MR1304906}
D.~Mumford, J.~Fogarty, and F.~Kirwan, \emph{Geometric invariant theory}, third ed., Ergebnisse der Mathematik und ihrer Grenzgebiete (2) [Results in Mathematics and Related Areas (2)], vol.~34, Springer-Verlag, Berlin, 1994. \MR{MR1304906 (95m:14012)}

\bibitem[Pol05]{MR2122731}
A.~Polishchuk, \emph{Noncommutative proj and coherent algebras}, Math. Res. Lett. \textbf{12} (2005), no.~1, 63--74. \MR{2122731}

\bibitem[{Sta}18]{stacks-project}
The {Stacks Project Authors}, \emph{\textit{Stacks Project}}, \url{https://stacks.math.columbia.edu}, 2018.

\bibitem[VdB01]{MR1846352}
Michel Van~den Bergh, \emph{Blowing up of non-commutative smooth surfaces}, Mem. Amer. Math. Soc. \textbf{154} (2001), no.~734, x+140. \MR{1846352 (2002k:16057)}

\end{thebibliography}

\end{document}